\documentclass[11pt, a4paper]{amsart}
\usepackage{amsfonts,amssymb,amsmath, amsthm}
\usepackage{mathrsfs}
\usepackage{lmodern}
\usepackage{qsymbols}
\usepackage{latexsym}
\usepackage[noadjust]{cite}
\usepackage{pdfsync}
\usepackage{bm}
\usepackage[cmtip,all]{xy}
\usepackage{enumitem}

\newtheorem{thm}{Theorem}[section]
\newtheorem{lem}[thm]{Lemma}
\newtheorem{prop}[thm]{Proposition}
\newtheorem{cor}[thm]{Corollary}
\theoremstyle{definition}
\newtheorem{defn}[thm]{Definition}
\newtheorem{rem}[thm]{Remark}
\numberwithin{equation}{section}

\usepackage[plainpages=false,pdfpagelabels,backref=page,
citecolor=red]{hyperref}
\usepackage{xcolor}
\hypersetup{
 colorlinks,
 linkcolor={cyan!90!black},
 citecolor={magenta},
 urlcolor={green!40!black}
} 
\newcommand{\bv}{\mathbf{v}}
\newcommand{\R}{\mathbb{R}}
\newcommand{\IC}{\mathbb{C}}

\newcommand{\IE}{\mathbb{E}}
\newcommand{\IW}{\mathbb{W}}
\newcommand{\cS}{\mathcal{S}} 
\newcommand{\cF}{\mathcal{F}} 

\newcommand{\cE}{\mathcal{E}}
\newcommand{\cP}{\mathcal{P}}
\newcommand{\cR}{\mathcal{R}}
\newcommand{\loc}{\operatorname{loc}}
\renewcommand{\L}{\operatorname{L}} 
\newcommand{\C}{\operatorname{C}} 
\renewcommand{\H}{\operatorname{H}} 
\newcommand{\W}{\operatorname{W}}
\newcommand{\B}{\operatorname{B}}

\newcommand{\dhalf}{D_t^{1/2}} 
\newcommand{\HT}{H_t} 
\newcommand{\e}{\mathrm{e}} 
\let\ii\i
\renewcommand{\i}{\mathrm{i}} 
\newcommand{\eps}{\varepsilon} 
\renewcommand\Re{\operatorname{Re}}

\newcommand{\blank}{\mathrel{\,\cdot\,}} 
\newcommand{\Lop}{\mathcal{L}} 
\newcommand{\cl}[1]{\overline{#1}} 
\DeclareMathOperator{\supp}{supp} 
\renewcommand{\d}{\mathrm{d}} 
\newcommand{\semi}[1]{[\mathrel{\,#1\,}]}
\newcommand{\Isemi}[1]{[\![\mathrel{\,#1\,}]\!]}
\newcommand{\dual}[2]{\langle #1,#2 \rangle}

\newcommand{\sgn}{\operatorname{sgn}}

\def\Xint#1{\mathchoice
{\XXint\displaystyle\textstyle{#1}}%
{\XXint\textstyle\scriptstyle{#1}}%
{\XXint\scriptstyle\scriptscriptstyle{#1}}%
{\XXint\scriptscriptstyle%
\scriptscriptstyle{#1}}%
\!\int}
\def\XXint#1#2#3{{\setbox0=\hbox{$#1{#2#3}{%
\int}$ }
\vcenter{\hbox{$#2#3$ }}\kern-.6\wd0}}
\def\barint{\,\Xint -} 
\def\bariint{\barint_{} \kern-.4em \barint}
\def\bariiint{\bariint_{} \kern-.4em \barint}
\renewcommand{\iint}{\int_{}\kern-.34em \int} 
\renewcommand{\iiint}{\iint_{}\kern-.34em \int} 
\title[Non-local self-improving properties]{Non-local self-improving properties:\\ A functional analytic approach}
\author{Pascal Auscher}
\author{Simon Bortz}
\author{Moritz Egert}
\author{Olli Saari}
\address{Laboratoire de Math\'{e}matiques d'Orsay, Univ. Paris-Sud, CNRS, Universit\'{e} Paris-Saclay, 91405 Orsay, France}
\email{pascal.auscher@math.u-psud.fr, moritz.egert@math.u-psud.fr}

\address{	
	School of Mathematics, University of Minnesota, Minneapolis, MN
55455, USA}
	\email{bortz010@umn.edu} 

\address{	
	Department of Mathematics and Systems Analysis, 
	Aalto University,
	FI-00076 Aalto, 
	Finland}
	\email{olli.saari@aalto.fi}

\thanks{The first and third authors were partially supported by the ANR project ``Harmonic Analysis at its Boundaries'', ANR-12-BS01-0013. This material is based upon work supported by National Science Foundation under Grant No.\ DMS-1440140 while the  authors were in residence at the MSRI in Berkeley, California, during the Spring 2017 semester. The second author was supported by the NSF INSPIRE Award DMS 1344235. The third author was supported by a public grant as part of the FMJH. The fourth author was supported by the Academy of Finland (Decision No.\ 277008).}

\date{\today}
\dedicatory{}
\begin{document}
\begin{abstract}
A functional analytic approach to obtaining self-improving properties of solutions to linear non-local elliptic equations is presented. It yields conceptually simple and very short proofs of some previous results due to Kuusi--Mingione--Sire and Bass--Ren. Its flexibility is demonstrated by new applications to non-autonomous parabolic equations with non-local elliptic part and questions related to maximal regularity.
\end{abstract}
\maketitle
\section{Introduction}

Recently, there has been a particular interest in linear elliptic integrodifferential equations of type
\begin{align*}
 \iint_{\R^n \times \R^n} A(x,y) \frac{(u(x) - u(y))\cdot \cl{(\phi(x)-\phi(y))}}{|x-y|^{n+2 \alpha}} \, \d x \, \d y = \int_{\R^n} f(x) \cdot \cl{\phi(x)} \, \d x \quad (\phi \in \C_0^\infty(\R^n)),
\end{align*}
where the kernel $A$ is a measurable function on $\R^n \times \R^n$ with bounds
\begin{align}
\label{eq:ellip}
0<\lambda \leq \Re A(x,y) \leq |A(x,y)| \leq \lambda^{-1} \qquad (\text{a.e. }(x,y) \in \R^n \times \R^n) 
\end{align}
and $\alpha$ is a number strictly between $0$ and $1$. See for example \cite{BR,BWZ1,BWZ2,KMS,LPPS,Schikorra}. Such fractional equations of order $2 \alpha$ exhibit new phenomena that do not have any counterpart in the theory of second order elliptic equations in divergence form: In \cite{KMS}, building on earlier ideas in \cite{BR}, it has been shown that under appropriate integrability assumptions on $f$, weak solutions $u$ in the corresponding fractional $\L^2$-Sobolev space $\W^{\alpha,2}(\R^n)$ self-improve in integrability \emph{and} in differentiability. Whereas the former is also known for second-order equations under the name of ``Meyers' estimate'' \cite{Meyers}, the improvement in regularity without any further smoothness assumptions on the coefficients is a feature of non-local equations only \cite[p.~59]{KMS}. We mention that \cite{KMS} also treats semi-linear variants of the equation above, but already the linear case is of interest for further applications, for example to the stability of stable-like processes \cite{BR}. 

Up to now, most approaches are guided by the classical strategy for the second-order case, that is, they employ fractional Caccioppoli inequalities to establish non-local reverse H\"older estimates and then prove a delicate self-improving property for such inequalities in the spirit of Gehring's lemma. The purpose of this note is to present a functional analytic approach which we believe is of independent interest for several other applications related to partial differential equations of fractional order as it yields short and conceptually very simple proofs. 

Let us outline our strategy that is concisely implemented in Section~\ref{sec:Dirichlet form}. Writing the fractional equation in operator form
\begin{align}
\label{eq:shorthand}
\langle \Lop_{\alpha,A}u , \phi \rangle = \langle f, \phi \rangle, \qquad (u,\phi \in \W^{\alpha,2}(\R^n)),
\end{align}
the left-hand side is associated with a sesquilinear form on the Hilbert space $\W^{\alpha,2}(\R^n)$ and thanks to ellipticity \eqref{eq:ellip} the Lax-Milgram lemma applies and yields invertibility of $1+\Lop_{\alpha,A}$ onto the dual space. Now, the main difference compared with second order elliptic equations is that we can transfer regularity requirements between $u$ and $\phi$ without interfering with the coefficients $A$:
Without making any further assumption we may write
\begin{align*}
 \langle \Lop_{\alpha,A}u,  \phi \rangle = \iint_{\R^n \times \R^n} A(x,y) \frac{u(x) - u(y)}{|x-y|^{n/2+\alpha + \eps}} \cdot \frac{\cl{\phi(x)-\phi(y)}}{|x-y|^{n/2+ \alpha-\eps}} \, \d x \, \d y,
\end{align*}
which yields boundedness $\Lop_{\alpha,A}: \W^{\alpha+\eps,2}(\R^n) \to \W^{\alpha-\eps,2}(\R^n)^*$. Then the ubiquitous analytic perturbation lemma of \v{S}ne{\u{\ii}}berg \cite{Sneiberg-Original} allows one to extrapolate invertibility to $\eps>0$ small enough. Compared to \cite{BR,KMS} we can also work in an $\L^p$-setting without hardly any additional difficulties. In this way, we shall recover some of their results on global weak solutions in Section~\ref{sec:higher diff} and discuss some new and sharpened local self-improvement properties in Section~\ref{sec:local}.

Finally, in Section~\ref{sec:parabolic} we demonstrate the simplicity and flexibility of our approach by proving that for each $f \in \L^2(0,T; \L^2(\R^n))$ the unique solution $u \in \H^1(0,T; \W^{\alpha,2}(\R^n)^*) \cap \L^2(0,T; \W^{\alpha,2}(\R^n)) $ of the non-autonomous Cauchy problem
\begin{align*}
 u'(t) + \Lop_{\alpha,A(t)} u(t) = f(t), \qquad u(0) = 0,
\end{align*}
self-improves to the class $\H^1(0,T; \W^{\alpha-\eps,2}(\R^n)^*) \cap \L^2(0,T; \W^{\alpha+\eps,2}(\R^n))$ for some $\eps>0$. Here, each $\Lop_{\alpha,A(t)}$ is a fractional elliptic operator as in \eqref{eq:shorthand} with uniform upper and lower bounds in $t$ but again we do not assume any regularity on $A(t,x,y) := A(t)(x,y)$ besides measurability in all variables. We remark that $\eps = \alpha$ and $\W^{0,2}(\R^n) := \L^2(\R^n)$ would mean maximal regularity, which in general requires some smoothness of the coefficients in the $t$-variable. See \cite{MaxRegSurvey} for a recent survey and the recent paper \cite{Grubb} for related results on regularity of solutions to such fractional heat equations with smooth coefficients. In this regard, our results reveal a novel phenomenon in the realm of non-autonomous maximal regularity. Let us remark that recently we have explored related techniques also for second-order parabolic systems in \cite{localpara}.
\section{Notation}
\label{sec:Notation}
Any Banach space $X$ under consideration is taken over the complex numbers and we shall denote by $X^*$ the \emph{anti-dual space} of conjugate linear functionals $X \to \IC$. In particular, all function spaces are implicitly assumed to consist of complex valued functions. Throughout, we assume the dimension of the underlying Euclidean space to be $n \geq 2$.

Given $s \in (0,1)$ and $p \in (1,\infty)$, the \emph{fractional Sobolev space} $\W^{s,p}(\R^n)$ consists of all $u \in \L^p(\R^n)$ with finite semi-norm
\begin{align*}
 \semi{u}_{s,p} := \bigg(\iint_{\R^n \times \R^n} \frac{|u(x) - u(y)|^p}{|x-y|^{n+sp}} \, \d x \, \d y \bigg)^{1/p} < \infty.
\end{align*}
It becomes a Banach space for the norm $\|\cdot\|_{s,p} :=  (\|\cdot\|_p^{p} + \semi{\blank}_{s,p}^{p})^{1/p}$, where here and throughout $\|\cdot\|_p$ denotes the norm on $\L^p(\R^n)$. Moreover, $\W^{s,2}(\R^n)$ is a Hilbert space for the inner product
\begin{align*}
 \langle u,v \rangle := \int_{\R^n} u(x) \cdot \cl{v(x)} \, \d x + \iint_{\R^n \times \R^n} \frac{(u(x) - u(y))\cdot \cl{(v(x)-v(y))}}{|x-y|^{n+2s}} \, \d x \, \d y.
\end{align*}
Every so often, it will be more convenient to view $\W^{s,p}(\R^n)$ within the scale of Besov spaces. More precisely, taking $\phi \in \cS(\R^n)$ with Fourier transform $\cF \phi: \R^n \to [0,1]$ such that $\cF \phi(\xi) = 1$ for $|\xi| \leq 1$ and $\cF \phi(\xi) = 0$ for $|\xi| \geq 2$ and defining $\phi_0:= \phi$ and $(\cF \phi_j)(\xi) := \cF \phi(2^{-j} \xi) - \cF \phi (2^{-j+1} \xi)$ for $\xi \in \R^n$ and $j \geq 1$, the \emph{Besov space} $\B^{s}_{p,p}(\R^n)$ is the collection of all $u \in \L^p(\R^n)$ with finite norm
\begin{align}
\label{BesovNorm}
 \|u\|_{\B^{s}_{p,p}(\R^n)} :=  \bigg(\sum_{j = 0}^\infty 2^{jsp} \|\phi_j \ast u\|_p^p \bigg)^{1/p} < \infty.
\end{align}
Different choices of $\phi$ yield equivalent norms on $\B^{s}_{p,p}(\R^n)$. Moreover, the Schwartz class $\cS(\R^n)$, and thus also the space of smooth compactly supported functions $\C_0^\infty(\R^n)$, is dense in any of these spaces, see \cite[Sec.~2.3.3]{NewTriebel}. Finally, $\W^{s,p}(\R^n) = \B_{p,p}^{s}(\R^n)$ up to equivalent norms \cite[Sec.~2.5.12]{NewTriebel}.
\section{Analysis of the Dirichlet form}
\label{sec:Dirichlet form}

In this section, we carefully analyze the mapping properties of the \emph{Dirichlet form}
\begin{align}
\label{eq:Lopa}
 \cE_{\alpha,A}(u,v) := \iint_{\R^n \times \R^n} A(x,y) \frac{(u(x) - u(y))\cdot \cl{(v(x)-v(y))}}{|x-y|^{n+2\alpha}} \, \d x \, \d y,
\end{align}
which we define here for $u,v \in \W^{\alpha,2}(\R^n)$. Starting from now, $\alpha \in (0,1)$ is fixed and $A: \R^n \times \R^n \to \IC$ denotes a measurable kernel that satisfies the accretivity condition \eqref{eq:ellip}. This entails boundedness 
\begin{align*}
 |\cE_{\alpha,A}(u,v)| \leq \lambda^{-1} \semi{u}_{\alpha,2} \semi{v}_{\alpha,2} \leq \lambda^{-1} \|u\|_{\alpha,2} \|v\|_{\alpha,2}
\end{align*}
and quasi-coercivity
\begin{align}
\label{eq:lower bound sesqui}
 \Re \cE_{\alpha,A}(u,u) \geq \lambda \semi{u}_{\alpha,2}^2  
 \geq \lambda \|u\|_{\alpha,2}^2 - \|u\|_2^2. 
\end{align}
Together with the sesquilinear form $\cE_{\alpha,A}$ comes the associated operator $\Lop_{\alpha,A}: \W^{\alpha,2}(\R^n) \to \W^{\alpha,2}(\R^n)^*$ defined through
\begin{align*}
 \langle \Lop_{\alpha,A}u, v \rangle := \cE_{\alpha,A}(u,v),
\end{align*}
where $\langle \cdot \,, \cdot \rangle$ denotes the sesquilinear duality between $\W^{\alpha,2}(\R^n)$ and its anti-dual, extending the inner product on $\L^2(\R^n)$. 

As an immediate consequence of the Lax-Milgram lemma we can record

\begin{lem}
\label{lem:Lax-Milgram}
The operator $1 + \Lop_{\alpha,A}: \W^{\alpha,2}(\R^n) \to \W^{\alpha,2}(\R^n)^*$ is bounded and invertible. Its norm and the norm of its inverse do not exceed $\lambda^{-1}$.
\end{lem}

The key step in our argument will be to obtain the analogous result on `nearby' fractional Sobolev spaces $\W^{s,p}(\R^n)$. We begin with boundedness, which of course is the easy part.

\begin{lem}
\label{lem:1+L bounded}
Let $s,s' \in (0,1)$ and $p,p' \in (1,\infty)$ satisfy $s + s' = 2 \alpha$ and $1/p + 1/p' = 1$. Then $1 + \Lop_{\alpha,A}$ extends from $\C_0^\infty(\R^n)$ by density to a bounded operator $\W^{s,p}(\R^n) \to \W^{s',p'}(\R^n)^*$ denoted also by $1+\Lop_{\alpha,A}$, and
\begin{align*}
  \big|\big \langle u+\Lop_{\alpha,A} u, v \big \rangle \big| \leq \|u\|_{p} \|v\|_{p'} + \lambda^{-1} \semi{u}_{s,p} \semi{v}_{s',p'} \qquad 
\end{align*}
for all $u \in \W^{s,p}(\R^n)$ and all $v \in \W^{s',p'}(\R^n)$.
\end{lem}

\begin{proof}
Given $u,v \in \W^{\alpha,2}(\R^n)$ we split $n + 2 \alpha = (n/p + s) + (n/p' + s')$ and apply H\"older's inequality with exponents $1 = 1/\infty + 1/p + 1/p'$ to give
\begin{align*}
 \big|\big \langle \Lop_{\alpha,A} u, v \big \rangle \big|
= \bigg|\iint_{\R^n \times \R^n} A(x,y) \frac{(u(x) - u(y))\cdot \cl{(v(x)-v(y))}}{|x-y|^{n+2\alpha}} \, \d x \, \d y \bigg|
\leq \lambda^{-1} \semi{u}_{s,p} \semi{v}_{s',p'}.
\end{align*}
Again by H\"older's inequality $|\langle u, v \rangle| \leq \|u\|_{p} \|v\|_{p'}$, yielding the required estimate for $u,v \in \W^{\alpha,2}(\R^n)$. Since $\C_0^\infty(\R^n)$ is a common dense subspace of all fractional Sobolev spaces under consideration here (see Section~\ref{sec:Notation}) this precisely means that $1+\Lop_{\alpha,A}$ extends to a bounded operator from $\W^{s,p}(\R^n)$ into the anti-dual space of $\W^{s',p'}(\R^n)$.
\end{proof}

\begin{rem}
\label{rem:1+L bounded}
It follows from Fatou's lemma that for $u$ and $v$ as in Lemma~\ref{lem:1+L bounded} we still have $\langle \Lop_{\alpha,A}u, v \rangle = \cE_{\alpha,A}(u,v)$ with the right-hand side given by \eqref{eq:Lopa}.
\end{rem}

We turn to the study of invertibility by means of a powerful analytic perturbation argument going back to {\v{S}}ne{\u{\ii}}berg~\cite{Sneiberg-Original}. In essence, the only supplementary piece of information needed for this approach is that the function spaces for boundedness obtained above form a complex interpolation scale.

We denote by $[X_0,X_1]_\theta$, $0<\theta<1$, the scale of \emph{complex interpolation spaces} between two Banach spaces $X_0$, $X_1$ that are both included in the tempered distributions $\cS'(\R^n)$. Background information can be found in \cite{Bergh-Loefstroem} and \cite{NewTriebel}, but for the understanding of this paper we do not require any further knowledge on this theory except for the identity
\begin{align}
\label{eq:interpolSob}
 \big[\W^{s_0,p_0}(\R^n),  \W^{s_1,p_1}(\R^n)\big]_\theta = \W^{s,p}(\R^n)
\end{align}
for $p_0, p_1 \in (1,\infty)$, $s_0,s_1 \in (0,1)$, with $p,s$ given by
\begin{align*}
 \frac{1}{p} = \frac{1-\theta}{p_0} + \frac{\theta}{p_1}, \qquad s = (1-\theta)s_0 + \theta s_1,
\end{align*}
and the analogous identity for the anti-dual spaces. Equality \eqref{eq:interpolSob} is in the sense of Banach spaces with equivalent norms and the equivalence constants are uniform for $s_i, p_i, \theta$ within compact subsets of the respective parameter intervals. This uniformity is implicit in most proofs and we provide references where they are either stated or can be read off particularly easily: This is \cite[Sec.~2.5.12]{NewTriebel} to identify $\W^{s,p}(\R^n) = \B^{s}_{p,p}(\R^n)$ up to equivalent norms, \cite[Thm.~6.4.5(6)]{Bergh-Loefstroem} for the interpolation and \cite[Cor.~4.5.2]{Bergh-Loefstroem} for the (anti-) dual spaces.

\begin{prop}
\label{prop:invert}
Let $s,s' \in (0,1)$ and $p,p' \in (1,\infty)$ satisfy $s + s' = 2 \alpha$ and $1/p + 1/p' = 1$. There exists $\eps > 0$, such that if $|\frac{1}{2} - \frac{1}{p}| < \eps$ and $|s-\alpha| < \eps$, then
\begin{align*}
 1 + \Lop_{\alpha,A} : \W^{s,p}(\R^n) \to \W^{s',p'}(\R^n)^*
\end{align*}
is invertible and the inverse agrees with the one obtained for $s=\alpha$, $p=2$ on their common domain of definition. Moreover, $\eps$ and the norms of the inverses depend only on $\lambda$, $n$, and $\alpha$.
\end{prop}

\begin{proof}
Consider the spaces $\W^{s,p}(\R^n)$ and $\W^{s',p'}(\R^n)^*$ as being arranged in the $(s,1/p)$-plane, where $p \in (1,\infty)$ but to make sense of our assumption we only consider parameters $s$ such that additionally $s' = 2\alpha -s \in (0,1)$.
By Lemma~\ref{lem:1+L bounded} we have boundedness 
\begin{align*}
  1 + \Lop_\alpha : \W^{s,p}(\R^n) \to \W^{s',p'}(\R^n)^*
\end{align*}
at every such $(s,1/p)$ and Lemma~\ref{lem:Lax-Milgram} provides invertibility at $(\alpha,1/2)$. 

Now, consider any line in the $(s,1/p)$-plane passing through $(\alpha,1/2)$ and take $(s_0,1/p_0)$, $(s_1,1/p_1)$ on opposite sides of $(\alpha, 1/2)$. Then \eqref{eq:interpolSob} precisely says that the scale of complex interpolation spaces between $\W^{s_0,p_0}(\R^n)$ and $\W^{s_1,p_1}(\R^n)$ corresponds (up to uniformly controlled equivalence constants) to the connecting line segment. The same applies to $\W^{s'_0,p'_0}(\R^n)^*$ and $\W^{s'_1,p'_1}(\R^n)^*$ on the segment connecting $(s_0',1/p_0')$ and $(s_1',1/p_1')$ through $(\alpha,1/2)$. 

According to {\v{S}}ne{\u{\ii}}berg's result, invertibility at the interior point $(\alpha,1/2)$ of this segment implies invertibility on an open surrounding interval whose radius around $(\alpha,1/2)$ depends only on upper and lower bounds at the center and the constants of norm equivalence, see \cite{Sneiberg-Original} or \cite[Thm.~1.3.25]{EigeneDiss} for a quantitative version. In particular, we can pick the same interval on every line segment as above and obtain $\eps>0$ with the required property. Finally, consistency of the inverses with the one computed at $(\alpha,1/2)$ is a general feature of complex interpolation \cite[Thm.~8.1]{KMM}. 
\end{proof}
\section{Weak solutions to elliptic non-local problems}
\label{sec:higher diff}

We are ready to use the abstract results obtained so far, to establish higher differentiability and integrability results for weak solutions $u \in \W^{\alpha,2}(\R^n)$ to elliptic non-local problems of the form
\begin{align}
\label{eq:nonloc}
 \Lop_{\alpha,A} u = \Lop_{\beta, B} g + f.
\end{align}
Here, $\Lop_{\alpha,A}$ is associated with the form $\cE_{\alpha,A}$ in \eqref{eq:Lopa}. In the same way, $\Lop_{\beta,B}$ is associated with
\begin{align*}
 \cE_{\beta,B}(g,v) := \iint_{\R^n \times \R^n} B(x,y) \frac{(g(x) - g(y))\cdot \cl{(v(x)-v(y))}}{|x-y|^{n+2 \beta}} \, \d x \, \d y,
\end{align*}
where starting from now, we fix $\beta \in (0,1)$ and $B \in \L^\infty(\R^n \times \R^n)$. Just like before, this guarantees that $\cE_{\beta,B}$ is a bounded sesquilinear form on $\W^{\beta,2}(\R^n)$ and hence that $\Lop_{\beta,B}$ is bounded from $\W^{\beta,2}(\R^n)$ into its anti-dual. However, we carefully note that we do neither assume a lower bound on $B$ nor any relation between $\alpha$ and $\beta$. In particular, $\beta > \alpha$ is allowed.

In the most general setup that is needed here, weak solutions are defined as follows.

\begin{defn}
\label{def:weak sol}
Let $f \in \L_{\loc}^1(\R^n)$ and $g \in \L_{\loc}^1(\R^n)$ such that $\cE_{\beta,B}(g,\phi)$ converges absolutely for every $\phi \in \C_0^\infty(\R^n)$. A function $u \in \W^{\alpha,2}(\R^n)$ is called \emph{weak solution} to \eqref{eq:nonloc} if 
\begin{align*}
 \cE_{\alpha,A}(u,\phi) = \cE_{\beta,B}(g,\phi) + \int_{\R^n} f \cdot \cl{\phi} \, \d x \qquad (\phi \in \C_0^\infty(\R^n)).
\end{align*}
\end{defn}

Suppose now that we are given a weak solution $u \in \W^{\alpha, 2}(\R^n)$. In order to invoke Proposition~\ref{prop:invert}, we write \eqref{eq:nonloc} in the form
\begin{align*}
 (1+ \Lop_{\alpha,A}) u = \Lop_{\beta, B} g + f + u.
\end{align*}
Hence, we see that higher differentiability and integrability for $u$, that is $u \in \W^{s,p}(\R^n)$ for some $s> \alpha$ and $p>2$, follows at once provided we can show $\Lop_{\beta, B} g + f + u \in \W^{s',p'}(\R^n)^*$ with $s'< \alpha$ and $p' < 2$ as in Proposition~\ref{prop:invert}. So, for the moment, our task is to work out the compatibility conditions on $u$, $f$, and $g$ to run this argument.

\subsection{Compatibility conditions for the right-hand side}
\label{subsec:compatibility}

The standing assumptions for all results in this section are $s' \in (0,1)$, $p \in (1,\infty)$ and $1/p + 1/p' =1 $. 

We begin by recalling the fractional Sobolev inequality, which will already take care of $u$ and $f$.

\begin{lem}[{\cite[Thm.~6.5]{FractionalSobolev}}]
\label{lem:Fractional Sobolev}
Suppose $s'p' < n$ and put $1/p'^* := 1/p' - s'/n$. Then
\begin{align*}
 \|v\|_{p'^*} \lesssim \semi{v}_{s',p'} \qquad (v \in \W^{s',p'}(\R^n)).
\end{align*}
In particular, $\W^{s',p'}(\R^n) \subset \L^{p'^*}(\R^n)$ and $\L^{p_*}(\R^n) \subset \W^{s',p'}(\R^n)^*$ with continuous inclusions, where $1/p_* := 1/p + s'/n$.
\end{lem}

As for $g$, a dichotomy between the cases $2 \beta \geq \alpha$ and $2 \beta < \alpha$ occurs. This reflects a dichotomy for the parameter $s'$, which typically is close to $\alpha$. In the first case, $2 \beta \geq \alpha$, we shall rely on

\begin{lem}
\label{lem:Embeddings Lb large}
If $2 \beta - s' \in (0,1)$ and $g \in \W^{2\beta - s',p}(\R^n)$, then
\begin{align*}
 |\langle \Lop_{\beta,B} g, v \rangle| \leq \|B\|_\infty \semi{g}_{2\beta-s',p} \semi{v}_{s',p'} \qquad (v \in \W^{s',p'}(\R^n)).
\end{align*}
\end{lem}

\begin{proof}
Write $n + 2 \beta = (n/p + 2\beta -s') + (n/p' + s')$ and note that
\begin{align*}
 |\langle \Lop_{\beta,B} g, v \rangle| 
 \leq \iint_{\R^n \times \R^n} \bigg|\frac{g(x) - g(y)}{|x-y|^{n/p + 2 \beta -s'}}\bigg| \bigg|\frac{v(x)-v(y)}{|x-y|^{n/p' + s'}}\bigg| |B(x,y)| \, \d x \, \d y.
\end{align*}
The claim follows from H\"older's inequality.
\end{proof}

The second case, $2 \beta < \alpha$, is slightly more complicated as we need the following embedding related to the fractional Laplacian $(-\Delta)^{\beta}$, see \cite{SamkoEtAl, NewTriebel}. For the reader's convenience and later reference we give a direct argument.

\begin{lem}
\label{lem:Mixed embedding}
Suppose $s' > 2 \beta$, $s'p' < n$, and put $\frac{1}{q'} := \frac{1}{p'} - \frac{s'-2\beta}{n}$. Then
\begin{align*}
 \left(\int_{\R^n} \bigg( \int_{\R^n} \frac{|v(x) - v(y)|}{|x-y|^{n + 2 \beta}} \, \d y \bigg)^{q'} \, \d x \right)^{1/q'} \lesssim \semi{v}_{s',p'} \qquad (v \in \W^{s',p'}(\R^n)).
\end{align*}
\end{lem}

\begin{proof}
Let $v \in \W^{s',p'}(\R^n)$ and put $1/p'^* := 1/p' - s'/n$ as in Lemma~\ref{lem:Fractional Sobolev}, so that
\begin{align*}
 \frac{1}{q'} = \frac{2 \beta}{s' p'} + \frac{s' - 2 \beta}{s'} \frac{1}{p'^*} :=  \frac{1}{r_1} + \frac{1}{r_2}.
\end{align*}
Note that our assumptions guarantee $p'^*, r_1, r_2 \in (1,\infty)$. Denote by $M$ the Hardy-Littlewood maximal operator defined for $f \in \L_{\loc}^1(\R^n)$ via
\begin{align*}
 Mf(x) := \sup_{B \ni x} \frac{1}{|B|} \int_{B} |f(y)| \, \d y \qquad (x \in \R^n),
\end{align*}
where the supremum runs over all balls $B \subset \R^n$ that contain $x$. We claim that it suffices to prove
\begin{align}
\label{eq1:Mixed embedding}
 \int_{\R^n} \frac{|v(x) - v(y)|}{|x-y|^{n + 2 \beta}} \, \d y 
\lesssim \bigg(\int_{\R^n} \frac{|v(x) - v(y)|^{p'}}{|x-y|^{n + s'p'}} \, \d y \bigg)^{1/r_1} Mv(x)^{1 - p'/r_1} \qquad (\text{a.e. }x \in \R^n).
\end{align}
Indeed, temporarily assuming \eqref{eq1:Mixed embedding}, we can take $\L^q$-norms in the $x$-variable and apply H\"older's inequality on the integral in $x$ with exponents $ 1/q' = 1/r_1 + 1/r_2$
to deduce
\begin{align*}
 \left(\int_{\R^n} \bigg( \int_{\R^n} \frac{|v(x) - v(y)|}{|x-y|^{n + 2 \beta}} \, \d y \bigg)^{q'} \, \d x \right)^{1/q'} 
\lesssim \semi{v}_{s',p'}^{p'/r_1} \|Mv \|_{p'^*}^{1-p'/r_1}.
\end{align*}
The claim follows since we have $\|Mv\|_{p'^*} \lesssim \|v\|_{p'^*} \lesssim \semi{v}_{s',p'}$ by the maximal theorem and Lemma~\ref{lem:Fractional Sobolev}. 

Now, in order to establish \eqref{eq1:Mixed embedding} we split the integral at $|x-y| = h(x)$, with $h(x)$ to be chosen later. Since $2 \beta - s' < 0$ by assumption, we can write $n + 2 \beta = n/p' + s' + n/p + (2\beta - s')$ and apply H\"older's inequality to give
\begin{align}
\label{eq2:Mixed embedding}
\begin{split}
 \int_{|x-y| \leq h(x)} \frac{|v(x) - v(y)|}{|x-y|^{n + 2 \beta}} \, \d y 
&\leq h(x)^{s'- 2 \beta} \bigg(\int_{|x-y| \leq h(x)} \frac{|v(x) - v(y)|^{p'}}{|x-y|^{n + s'p'}} \, \d y \bigg)^{1/p'}
\\ & \leq h(x)^{s'- 2 \beta} \bigg(\int_{\mathbb{R}^n} \frac{|v(x) - v(y)|^{p'}}{|x-y|^{n + s'p'}} \, \d y \bigg)^{1/p'}.
\end{split}
\end{align}
The remaining integral is bounded by
\begin{align*}
 \int_{|x-y| \geq h(x)} \frac{|v(x) - v(y)|}{|x-y|^{n + 2 \beta}} \, \d y 
\leq \int_{|x-y| \geq h(x)} \frac{|v(x)|}{|x-y|^{n+2\beta}} \, \d y + \int_{|x-y| \geq h(x)} \frac{|v(y)|}{|x-y|^{n+2\beta}} \, \d y,
\end{align*}
where the first term equals $c|v(x)| h(x)^{-2 \beta}$ for some dimensional constant $c$. Next, on writing 
\begin{align*}
 \frac{1}{|x-y|^{n+2 \beta}} = \int_{|x-y|}^\infty \frac{n+2\beta}{r^n} \, \frac{\d r}{r^{1+2 \beta}}
\end{align*}
and changing the order of integration, the second term above becomes
\begin{align*}
 (n+ 2 \beta) \int_{h(x)}^\infty \bigg(\frac{1}{r^n}\int_{h(x) 
 \leq |x-y| \leq r} |v(y)| \, \d y \bigg) \, \frac{\d r}{r^{1+2 \beta}}
\end{align*}
and thus can be controlled by $C_{n,\beta}Mv(x) h(x)^{-2 \beta}$. Since $|v| \leq Mv$ almost everywhere, we obtain in conclusion
\begin{align}
\label{eq3:Mixed embedding}
 \int_{|x-y| \geq h(x)} \frac{|v(x) - v(y)|}{|x-y|^{n + 2 \beta}} \, \d y \lesssim h(x)^{-2 \beta} Mv(x) \qquad (\text{a.e. }x \in \R^n).
\end{align}
Finally, we pick $h(x)$ such that the right-hand sides of \eqref{eq2:Mixed embedding} and \eqref{eq3:Mixed embedding} are equal and obtain \eqref{eq1:Mixed embedding}.
\end{proof}

As an easy consequence we obtain the required bounds for $\Lop_{\beta,B}$.

\begin{cor}
\label{cor:Embdedding Lb small}
Suppose $s' > 2 \beta$, $s'p' < n$, and put $\frac{1}{q} := \frac{1}{p} + \frac{s'-2\beta}{n}$. For every $g \in \L^{q}(\R^n)$ there holds
\begin{align*}
 |\langle \Lop_{\beta,B} g, v \rangle| \lesssim \|B\|_\infty \|g\|_{q} \semi{v}_{s',p'} \qquad (v \in \W^{s',p'}(\R^n)).
\end{align*}
\end{cor}

\begin{proof}
We crudely bound $|g(x) - g(y)| \leq |g(x)| + |g(y)|$ in the integral representation for $\langle \Lop_{\beta,B} g, v \rangle$ and apply Tonelli's theorem to give
\begin{align*}
 |\langle \Lop_{\beta,B} g, v \rangle| 
 &\leq \int_{\R^n} |g(x)| \bigg(\int_{\R^n} \frac{|v(x) - v(y)|}{|x-y|^{n + 2 \beta}} \cdot (|B(x,y)| + |B(y,x)|) \, \d y \bigg) \, \d x \\
 &\leq 2 \|B\|_\infty \|g\|_{q} \left(\int_{\R^n} \bigg( \int_{\R^n} \frac{|v(x) - v(y)|}{|x-y|^{n + 2 \beta}} \, \d y \bigg)^{q'} \, \d x \right)^{1/q'},
\end{align*}
the second step being due to H\"older's inequality. Since the H\"older conjugate of $q$ is the exponent $q'$ appearing in Lemma~\ref{lem:Mixed embedding}, the claimed inequality follows from that very lemma.
\end{proof}

\subsection{Proof of a global higher differentiability and integrability result}

Combining Proposition~\ref{prop:invert} with the mapping properties found in the previous section, we can prove our main self-improvement property for weak solutions of \eqref{eq:nonloc} . As in \cite{KMS}, we impose the additional restriction $2 \beta - \alpha < 1$ in the case that $\beta > \alpha$.

\begin{thm}
\label{thm:main2}
There exists $\eps > 0$, depending only on $\lambda, n, \alpha, \beta$ with the following property. Suppose $s \in (\alpha,1)$ and $p \in [2,\infty)$ satisfy $|s-\alpha|, |p-2| < \eps$.
If $u \in \W^{\alpha,2}(\R^n)$ is a weak solution to \eqref{eq:nonloc}, then the following conditions guarantee $u \in \W^{s,p}(\R^n)$:
\begin{align*}
 f \in \L^r(\R^n), \qquad \frac{1}{r} = \frac{1}{p} + \frac{2\alpha -s}{n}
\end{align*}
and 
\begin{align*}
 g \in \L^q(\R^n), \qquad \frac{1}{q} = \frac{1}{p} + \frac{2 \alpha - 2 \beta -s}{n} \quad \text{if $2 \beta < \alpha$},
\end{align*}
or
\begin{align*}
 g \in \W^{2 \beta - 2 \alpha + s,p}(\R^n) \quad \text{if $0 \leq 2\beta - \alpha < 1$}.
\end{align*}
Moreover, there is an estimate
\begin{align*}
 \|u\|_{s,p} \lesssim \|u\|_{\alpha,2} + \|f\| + \|g\|,
\end{align*}
where the norms of $f$ and $g$ are with respect to the function spaces specified above and the implicit constant depends on $\lambda, n, \alpha, \beta,s,p$ and $ \|B\|_{\infty}$.
\end{thm}

\begin{proof}
As usual we write $s+s' = 2 \alpha$ and $1/p + 1/p' = 1$. We let $\eps>0$ as given by Proposition~\ref{prop:invert}. If we can show $\Lop_{\beta, B} g + f + u \in \W^{s',p'}(\R^n)^*$, upon possibly forcing further restrictions on $\eps$, then by density of $\C_0^\infty(\R^n)$ in the fractional Sobolev spaces we can write the equation for $u$ in the form
\begin{align*}
 (1+ \Lop_{\alpha,A}) u = \Lop_{\beta, B} g + f + u
\end{align*}
and Proposition~\ref{prop:invert} yields $u \in \W^{s,p}(\R^n)$ with bound
\begin{align}
\label{eq1:main2}
 \|u\|_{s,p} \lesssim  \|\Lop_{\beta, B} g + f + u\|_{\W^{s',p'}(\R^n)^*}.
\end{align}

By assumption and Lemma~\ref{lem:Fractional Sobolev} we have $u \in \L^p(\R^n)$ for all $p \in [2,2^*]$ with $1/2^* = 1/2 - \alpha/n$. Note that here we used our assumption $n \geq 2$. For $p$ in this range we write $1/p = (1-\theta)/2 + \theta/2^*$ with $\theta \in (0,1)$ and get for any $s' \in (0,1)$ the bound
\begin{align}
\label{eq2:main2}
 \|u\|_{\W^{s',p'}(\R^n)^*} \leq \|u\|_{p} \leq \|u\|_2^{1-\theta} \|u\|_{2^*}^\theta \lesssim \|u\|_{\alpha,2},
\end{align}
where the second step follows from H\"older's inequality. Next, we have $s'p' < 2 \alpha < 2 \leq n$ (since $s' < \alpha$ and $p' < 2$) and hence Lemma~\ref{lem:Fractional Sobolev} yields $\|f\|_{\W^{s',p'}(\R^n)^*} \lesssim \|f\|_r$. Finally, we consider $\Lop_{\beta,B} g$. 

Suppose first that $2 \beta < \alpha$. Upon taking $\eps$ smaller, we can assume $2 \beta < s'$, in which case $\|\Lop_{\beta,B} g \|_{\W^{s',p'}(\R^n)^*} \lesssim \|g\|_q$ follows from Corollary~\ref{cor:Embdedding Lb small}. If, on the other hand, $2 \beta - \alpha \in [0,1)$, then we can additionally assume $2 \beta -s' \in (0,1)$ and apply Lemma~\ref{lem:Embeddings Lb large} to give $\|\Lop_{\beta,B} g \|_{\W^{s',p'}(\R^n)^*} \lesssim \|g\|_{2 \beta - 2 \alpha + s,p}$. Inserting these estimates on the right-hand side of \eqref{eq1:main2} yields the desired bound for $u$.
\end{proof}

\subsection{Comparison to earlier results}

As a consequence of our method, the exponents $s$ and $p$ for the higher differentiability and integrability of $u$ in Theorem~\ref{thm:main2} are precisely related to the assumptions on $f$ and $g$. As far as more qualitative results are concerned, this is by no means necessary since the following fractional Sobolev embedding allows for some play with the exponents.

\begin{lem}[{\cite[Thm.~6.2.4/6.5.1]{Bergh-Loefstroem}}]
\label{lem:Fractional Sobolev 2}
Let $s_0,s_1, s_2 \in (0,1)$ and $1 < p_0 \leq p_1 < \infty$ satisfy $s_0 - n/p_0 = s_1 - n/p_1$ and $s_2 < s_1$. Then
\begin{align*}
 \W^{s_0, p_0}(\R^n) \subset \W^{s_1, p_1}(\R^n) \subset \W^{s_2, p_1}(\R^n)
\end{align*}
with continuous inclusions.
\end{lem}

As a particular example, we obtain a self-improving property more in the spirit of \cite[Thm.~1.1]{KMS}. For this we define the following exponents related to fractional Sobolev embeddings, see Lemma~\ref{lem:Fractional Sobolev}, 
\begin{align}
\label{SobExp2}
 2_{*,\alpha}:= \frac{2n}{n+2 \alpha}, \qquad 2_{*,\alpha-2\beta}:= \frac{2n}{n+2(\alpha-2\beta)},
\end{align}
where the second one will of course only be used when $2 \beta < \alpha$.

\begin{cor}
\label{thm:main}
Let $u \in \W^{\alpha,2}(\R^n)$ be a weak solution to \eqref{eq:nonloc}. Suppose for some $\delta > 0$ there holds $f \in \L^{2_{*,\alpha}+\delta}(\R^n) \cap \L^{2_{*,\alpha}}(\R^n)$ and
\begin{align*}
 g \in \begin{cases}
        \L^{2_{*,\alpha - 2 \beta} + \delta}(\R^n) \cap \L^{2_{*,\alpha - 2 \beta}}(\R^n) & \text {if $2\beta < \alpha$,} \\
        \W^{2\beta - \alpha + \delta,2}(\R^n) & \text{if  $0 \leq 2 \beta - \alpha <1$.}
       \end{cases}
\end{align*}
Then $u \in \W^{s,p}(\R^n)$ for some $s>\alpha$, $p>2$. Moreover, $s$ and $p$ depend only on $\lambda, n, \alpha, \beta$.
\end{cor}

\begin{proof}[Proof of Corollary~\ref{thm:main}]
Throughout, we will have $s \in (\alpha,1)$ and $p \in [2,\infty)$. We consider the case $2 \beta < \alpha$ first. By the log-convexity of the Lebesgue space norms we may lower the value $\delta>0$ as we please and still have the respective assumptions on $f$ and $g$. On the other hand, the exponents in Theorem~\ref{thm:main2} satisfy $r>2_{*,\alpha}$ and $q>2_{*,\alpha-2\beta}$ and in the limits $s \to \alpha$ and $p \to 2$ we get equality. Hence, we can apply Theorem~\ref{thm:main2} with some choice of $s>\alpha$ and $p>2$ and the claim follows. 

It remains to deal with the assumption on $g$ in the case $2\beta - \alpha \in [0,1)$. But according to Lemma~\ref{lem:Fractional Sobolev 2} we can find $s>\alpha$ and $p>2$ arbitrarily close to $\alpha$ and $2$, respectively, such that $\W^{2\beta - \alpha + \delta,2}(\R^n) \subset \W^{2 \beta - 2\alpha + s,p}(\R^n)$ holds with continuous inclusion and again $u \in \W^{s,p}(\R^n)$ follows by Theorem~\ref{thm:main2}.
\end{proof}

As another application we reproduce the main result in \cite{BR} concerning the non-local elliptic equation
\begin{align*}
 \Lop_{\alpha,A} u = f
\end{align*}
with $f \in \L^2(\R^n)$. We note that this corresponds to taking $g=0$ in the general equation \eqref{eq:nonloc}. Hence, the entire Section~\ref{subsec:compatibility} could be skipped except for the first lemma, thereby making the argument up to this stage particularly simple.

\begin{cor}
\label{cor:Bass-Ren}
Let $f \in \L^2(\R^n)$ and let $u \in \W^{\alpha,2}(\R^n)$ be a weak solution to $\Lop_{\alpha,A} u = f$. Then
\begin{align*}
 \Gamma u(x) := \bigg( \int_{\R^n} \frac{|u(x)-u(y)|^2}{|x-y|^{n+2\alpha}} \, \d y \bigg)^{1/2}
\end{align*}
satisfies for some $p>2$ and a constant $c$ both depending only on $\lambda, n, \alpha$,
\begin{align*}
 \|\Gamma u \|_p \leq c(\|u\|_2 + \|f\|_2).
\end{align*}
\end{cor}

\begin{proof}
We use the notation introduced in Theorem~\ref{thm:main2} and write as usual $s+s' = 2 \alpha$, $1/p + 1/p' = 1$. According to Lemma~\ref{lem:Fractional Sobolev} we have $\L^r(\R^n) \subset \W^{s',p'}(\R^n)^*$ with continuous inclusion and if $s$ and $p$ are sufficiently close to $\alpha$ and $2$, respectively, then we have $r < 2$. Obviously, we also have $\L^{p}(\R^n) \subset \W^{s',p'}(\R^n)^*$ and $p > 2$. Hence, by virtue of the splitting 
\begin{align*}
 f = f \cdot \mathbf{1}_{\{|f| < \|f\|_2\}} + f \cdot \mathbf{1}_{\{|f| \geq \|f\|_2\}} \in \L^{p}(\R^n) + \L^{r}(\R^n)
\end{align*}
we obtain $f \in \W^{s',p'}(\R^n)^*$ with bound $\|f\|_{\W^{s',p'}(\R^n)^*} \lesssim \|f\|_2$. Here $\mathbf{1}_E$ denotes the indicator function of the set $E \subset \R^n$. Moreover, $\|u\|_{\W^{s',p'}(\R^n)^*} \lesssim \|u\|_{\alpha,2}$, see \eqref{eq2:main2}, and thus we can follow the first part of the proof of Theorem~\ref{thm:main2} in order to find $s>\alpha$, $p>2$, and implicit constants depending only on the above mentioned parameters, such that
\begin{align*}
 \|u\|_{s,p} \lesssim  \|f\|_2 + \|u\|_{\alpha,2}.
\end{align*}
The pair $(s,p)$ could be chosen anywhere in the $(s,p)$-plane close to $(\alpha,2)$ but for a reason that will become clear later one, we shall impose the relation 
\begin{align}
\label{eq1:Bass-Ren}
 \frac{n}{2} - \frac{n}{p} = s- \alpha.
\end{align}
Quasi-coercivity of the form associated with $\Lop_{\alpha,A}$ along with the equation for $u$ yield
\begin{align*}
 \lambda  \semi{u}_{\alpha,2}^2 \leq |\cE_{\alpha,A}(u,u)| = \bigg|\int_{\R^n} f \cdot \cl{u} \, \d x \bigg| \leq \frac{1}{2}(\|u\|_2^2 + \|f\|_2^2),
\end{align*}
and thus it suffices to prove the estimate $\|\Gamma u\|_p \lesssim \|u\|_{s,p}$ to conclude. 

To this end, we split $\Gamma u(x) = \Gamma_1 u (x) + \Gamma_2 u (x)$ according to whether or not $|x-y| > 1$ in the defining integral. Repeating the argument to deduce \eqref{eq3:Mixed embedding}, we obtain
\begin{align*}
 |\Gamma_1 u(x)| = \bigg( \int_{|x-y|>1} \frac{|u(x) - u(y)|^2}{|x-y|^{n + 2 \alpha}} \, \d y \bigg)^{1/2} \lesssim M(|u|^2)(x)^{1/2}
\end{align*}
and as $p>2$, we conclude $\|\Gamma_1 u\|_p \lesssim \|u\|_p$ from the boundedness of the maximal operator on $\L^{p/2}(\R^n)$. As for the other piece, we use H\"older's inequality with exponent $p/2$ on the integral in $y$, to give
\begin{align*}
 \|\Gamma_2\|_p 
 &\lesssim \bigg(\int_{\R^n} \int_{|x-y|<1} \frac{|u(x) - u(y)|^p}{|x-y|^{np/2 + p \alpha}} \, \d y \, \d x \bigg)^{1/p} \leq [u]_{s,p},
\end{align*}
where in the final step we used that $np/2 + p \alpha = n + sp$ holds thanks to \eqref{eq1:Bass-Ren}.
\end{proof}
\section{Local results}
\label{sec:local}

In Theorem~\ref{thm:main2} and Corollary~\ref{thm:main}, we have obtained global improvements of regularity for solutions to \eqref{eq:nonloc} under global assumptions on the right-hand side. We now discuss some local analogues of this phenomenon. In order to formulate our main result in this direction, we define for balls $B \subset \R^n$ a local version of the fractional Sobolev norm by 
\begin{align*}
 \|u\|_{\W^{s,p}(B)}:= \bigg(\int_{B} |u(x)|^p \, \d x \bigg)^{1/p} + \bigg(\iint_{B \times B} \frac{|u(x) - u(y)|^p}{|x-y|^{n+sp}} \, \d x \, \d y \bigg)^{1/p} 
\end{align*}
and write $u \in \W^{s,p}(B)$ provided this quantity is finite. 

\begin{thm}
\label{thm:local}
There exists $\eps > 0$, depending only on $\lambda, n, \alpha, \beta$ with the following property. Suppose $s \in (\alpha,1)$ and $p \in [2,\infty)$ satisfy $|s-\alpha|, |p-2| < \eps$. Let $u \in \W^{\alpha,2}(\R^n)$ be a weak solution to \eqref{eq:nonloc} and let $B \subset \R^n$ be a ball. Then the following conditions guarantee $u \in \W^{s,p}(B')$ for every ball $B' \subset \subset B$:
\begin{align*}
 f \in \L^r(B) \quad \text{for some $r$ with} \quad \frac{1}{r} \leq \frac{1}{p} + \frac{2\alpha -s}{n}
\end{align*}
and 
\begin{align*}
 g \in \L^q(B) \cap \L^t(\R^n) \quad \text{for some $q,t$ with} \quad \frac{1}{q} \leq \frac{1}{p} + \frac{2 \alpha - 2 \beta -s}{n}, \quad \frac{1}{p} \leq \frac{1}{t} < \frac{1}{p} + \frac{2 \alpha -s}{n} \quad \text{if $2 \beta < \alpha$},
\end{align*}
or
\begin{align*}
 g \in \W^{2 \beta - 2 \alpha + s,p}(\R^n) \quad \text{if $0 \leq 2\beta - \alpha < 1$}.
\end{align*}
\end{thm}

Again, this gives a precise relation in the exponents, but we also state a more quantitative version. It follows by the exact same reasoning as Corollary~\ref{thm:main} was obtained from Theorems~\ref{thm:main2} in the previous section and we shall not provide further details. We are using again the lower Sobolev conjugates defined in \eqref{SobExp2}.

\begin{cor}
\label{thm:local2}
Let $u \in \W^{\alpha,2}(\R^n)$ be a weak solution to \eqref{eq:nonloc} and let $B \subset \R^n$ be a ball. Suppose for some $\delta > 0$ there holds $f \in \L^{2_{*,\alpha}+\delta}(B)$ and
\begin{align*}
 g \in \begin{cases}
        \L^{2_{*,\alpha - 2 \beta} + \delta}(B) \cap \L^{t}(\R^n) \quad \text{for some $t \in (2_{*,\alpha},2]$}& \text {if $2\beta < \alpha$,} \\
        \W^{2\beta - \alpha + \delta,2}(\R^n) & \text{if  $0 \leq 2\beta - \alpha < 1$.}
       \end{cases}
\end{align*}
Then there exist $s>\alpha$, $p>2$, such that $u \in \W^{s,p}(B')$ for every ball $B' \subset \subset B$. Moreover, $s$ and $p$ depend only on $\lambda, n, \alpha, \beta$.
\end{cor}

These statements are astonishingly local in that the assumption on $f$ and part of that for $g$ are only on the ball where we want to improve the regularity of $u$. To the best of our knowledge this has not been noted before. In particular, if $f$ and $g$ satisfy the assumption for every ball $B$, then the conclusion for $u$ holds for every ball $B'$. This is the result in \cite{KMS}. (Except that they suppose global integrability of exponent $t=2_{*,\alpha - 2 \beta} + \delta$ instead, which for large $\delta$ is not comparable with the condition in Corollary \ref{thm:local2}. It is possible to modify our argument to work in the setting of \cite{KMS} as well, but we leave this extension to interested reader, see Remark \ref{remark:localmod}.)

For the proof of Theorem~\ref{thm:local} it is instructive to recall a simple connection between the condition $\chi u \in \W^{s,p}(\R^n)$ for some $\chi \in \C_0^\infty(B)$ and the fractional Sobolev norm $\|\cdot\|_{\W^{s,p}(B)}$: On the one hand, denoting by $d>0$ the distance between the support of $\chi$ and ${}^cB$ we obtain from the mean value theorem,
\begin{equation}\label{cutoffsob1.eq}
\begin{split}
\bigg(\iint_{\R^n \times \R^n}  \frac{| (\chi u)(x) - (\chi u)(y)|^p}{|x-y|^{n+sp}} \, \d x \, \d y \bigg)^{1/p}
&\leq 2 \|\chi\|_\infty \bigg(\iint_{B \times B} \frac{|u(x) - u(y)|^p}{|x-y|^{n+sp}} \, \d x \, \d y \bigg)^{1/p} \\
&\quad + 4 \|\chi\|_\infty \bigg(\int_B |u(x)|^p \bigg(\int_{|x-y| \geq d} \frac{1}{|x-y|^{n+sp}} \, \d y\bigg) \, \d x \bigg)^{1/p} \\
&\quad + 2 \|\nabla \chi\|_\infty \bigg(\int_B |u(x)|^p \bigg(\int_B \frac{1}{|x-y|^{n+(s-1)p}} \, \d y\bigg) \, \d x \bigg)^{1/p},
\end{split}
\end{equation} 
where by symmetry and the fact that the integrand is zero when $x, y \not\in \supp(\chi)$, we can assume $x \in \supp(\chi)$ and then distinguish whether or not $y \in B$. As $s>0$ and $s-1<0$, the second and third terms are finite. Hence, we see that $u \in \W^{s,p}(B)$ implies $\chi u \in \W^{s,p}(\R^n)$. On the other hand, if $\chi = 1$ on a smaller ball $B' \subset \subset B$, then 
\begin{equation}\label{cutoffsob2.eq}
 \bigg(\int_{B'} |u(x)|^p \, \d x\bigg)^{1/p} + \bigg(\iint_{B' \times B'} \frac{|u(x) - u(y)|^p}{|x-y|^{n+sp}} \, \d x \, \d y \bigg)^{1/p} \leq \|\chi u\|_{s,p}.
\end{equation}
Due to these observations and the fact that Lebesgue spaces on a ball are ordered by inclusion, we see that Theorem~\ref{thm:local} follows at once from

\begin{lem}
\label{lem:local}
There exists $\eps > 0$, depending only on $\lambda, n, \alpha, \beta$ with the following property. Suppose $s \in (\alpha,1)$ and $p \in [2,\infty)$ satisfy $|s-\alpha|, |p-2| < \eps$. Let $u \in \W^{\alpha,2}(\R^n)$ be a weak solution to \eqref{eq:nonloc} and let $\chi \in \C_0^\infty(\R^n)$. Assume
\begin{align*}
 \chi f \in \L^r(\R^n) \quad \text{with} \quad \frac{1}{r} = \frac{1}{p} + \frac{2\alpha -s}{n}
\end{align*}
and if $2 \beta < \alpha$ assume
\begin{align*}
 \chi g \in \L^q(\R^n), \quad \frac{1}{q} = \frac{1}{p} + \frac{2 \alpha - 2 \beta -s}{n}, \quad \text{and} \quad g \in \L^t(\R^n), \quad \frac{1}{p} \leq \frac{1}{t} < \frac{1}{p} + \frac{2 \alpha - s}{n},
\end{align*}
whereas if $0 \leq 2\beta - \alpha < 1$ assume $g \in \W^{2 \beta - 2 \alpha + s,p}(\R^n)$. Then $\chi u \in \W^{s,p}(\R^n)$.
\end{lem}

The strategy for the proof of this key lemma is as follows. We let $u \in \W^{\alpha,2}(\R^n)$ be a weak solution to \eqref{eq:nonloc} and seek to write down a related fractional equation for $\chi u$ in order to be able to apply Proposition~\ref{prop:invert}. To this end, we note for three functions $u, \chi, \phi$ and $x,y \in \R^n$ the factorization
\begin{align}
\label{eq:factorization}
\begin{split}
&(\chi_x u_x - \chi_y u_y)(\phi_x - \phi_y) \\
&= (\chi_x \phi_x - \chi_y \phi_y)(u_x - u_y) + u_y(\chi_x - \chi_y)\phi_x + u_x (\chi_y-\chi_x) \phi_y \\
&= (\chi_x \phi_x - \chi_y \phi_y)(u_x - u_y) - (u_x - u_y) (\chi_x - \chi_y) \phi_y + u_y (\chi_x - \chi_y) (\phi_x - \phi_y),
\end{split}
\end{align}
where $u_x := u(x)$ and so on for the sake of readability. This identity plugged into the definition of $\cE_{\alpha,A}$, see \eqref{eq:Lopa}, yields
\begin{align*}
 \langle \Lop_{\alpha,A}(\chi u), \phi \rangle 
 &= \langle \Lop_{\alpha,A}u, \chi \phi \rangle + \langle \cR_{\alpha,A,\chi} u,\phi \rangle \qquad (\phi \in \C_0^\infty(\R^n)),
\end{align*}
where
\begin{align*}
\begin{split}
 \langle  \cR_{\alpha,A,\chi} u,\phi \rangle 
&:= - \iint_{\R^n \times \R^n}  A(x,y) \frac{(u(x)-u(y)) \cdot (\chi(x) - \chi(y))}{|x-y|^{n+2\alpha}} \cl{\phi(y)} \, \d x \, \d y \\
& \qquad   +  \iint_{\R^n \times \R^n} A(x,y) u(y) \frac{ (\chi(x) - \chi(y)) \cdot \cl{(\phi(x) - \phi(y))}   }{|x-y|^{n+2\alpha}}   \, \d x \, \d y
\end{split}
\end{align*}
provided all integrals are absolutely convergent. We shall check that in the proofs below. Of course, a similar calculation applies to $\Lop_{\beta,B}$. Therefore $\chi u \in \W^{\alpha,2}(\R^n)$ solves the non-local elliptic equation
\begin{align}
\label{eq:localized equation}
(1+\Lop_{\alpha,A})(\chi u) = \cR_{\alpha,A,\chi} u - \cR_{\beta,B,\chi} g + \chi u + \Lop_{\beta,B}(\chi g) + \chi f
\end{align}

With this strategy in place, we turn to the

\begin{proof}[Proof of Lemma~\ref{lem:local}]
We start by taking $\eps > 0$ as provided by Theorem~\ref{thm:main2} but for some steps we possibly need to impose additional smallness conditions that depend upon $n,\alpha,\beta$ through fractional Sobolev embeddings. As usual, we write $s+s' = 2 \alpha$ and $1/p + 1/p' = 1$.

The claim is $\chi u \in \W^{s,p}(\R^n)$ and according to Proposition~\ref{prop:invert} we only need to make sure that the right-hand side in \eqref{eq:localized equation} belongs to $\W^{s',p'}(\R^n)^*$. But from the proof of Theorem~\ref{thm:main2} we know that this is the case for $\chi u \in \W^{\alpha,2}(\R^n)$ and that the conditions on $\chi f$ and $\chi g$ are designed to make it work for the last two terms. 


We are left with the error terms. We start with $\cR_{\alpha,A,\chi}$, which as we recall is given for $\phi \in \C_0^\infty(\R^n)$ by
\begin{align*}
\begin{split}
 \langle  \cR_{\alpha,A,\chi} u,\phi \rangle
&:= - \iint_{\R^n \times \R^n}  A(x,y) \frac{(u(x)-u(y)) \cdot (\chi(x) - \chi(y))}{|x-y|^{n+2\alpha}} \cl{\phi(y)} \, \d x \, \d y \\
& \qquad   +  \iint_{\R^n \times \R^n} A(x,y) u(y) \frac{ (\chi(x) - \chi(y)) \cdot \cl{(\phi(x) - \phi(y))}   }{|x-y|^{n+2\alpha}}   \, \d x \, \d y \\
&:= I + II.
\end{split}
\end{align*}
Now,
\begin{align}
\label{eq1:local}
 \int_{\R^n} \frac{|\chi(x) - \chi(y)|^{p}}{|x-y|^{n+sp}} \, \d x \leq \int_{|x-y| \geq 1} \frac{ 2^{p} \|\chi\|_\infty^{p}}{|x-y|^{n + sp}} \, \d x + \int_{|x-y| < 1} \frac{\|\nabla \chi\|_\infty^{p}}{|x-y|^{n + (s-1)p}} \, \d x \lesssim 1
\end{align}
uniformly in $y \in \R^n$ since $s < 1$. Thus, applying H\"older's inequality first in $x$ and then in $y$, we obtain
\begin{align*}
 |II| 
 \leq \lambda^{-1} \int_{\R^n} |u(y)| \bigg( \int_{\R^n} \frac{|\chi(x) - \chi(y)|^{p}}{|x-y|^{n+sp}} \, \d x \bigg)^{1/p} \bigg(\int_{\R^n} \frac{|\phi(x) - \phi(y)|^{p'}}{|x-y|^{n+s'p'}} \, \d x \bigg)^{1/p'} \, \d y
 \lesssim \|u\|_p \semi{\phi}_{s',p'}.
\end{align*}
Similarly, but reversing the roles of $\phi$ and $u$, we get
\begin{align*}
 |I| 
 \leq \lambda^{-1} \int_{\R^n} |\phi(y)| \bigg( \int_{\R^n} \frac{|\chi(x) - \chi(y)|^{2}}{|x-y|^{n+ 2\alpha}} \, \d x \bigg)^{1/2} \bigg(\int_{\R^n} \frac{|u(x) - u(y)|^{2}}{|x-y|^{n+2\alpha}} \, \d x \bigg)^{1/2} \, \d y
 \lesssim \semi{u}_{\alpha,2} \| \phi \|_{2}.
\end{align*}
By making $\eps>0$ smaller, we can assume $1/2 - \alpha/n \leq 1/p$ and $1/{p'} - s'/n \leq 1/2$, which pays for continuous inclusion $\W^{\alpha,2}(\R^n) \subset \L^p(\R^n)$ and $\W^{s',p'}(\R^n) \subset \L^2(\R^n)$, see Lemma~\ref{lem:Fractional Sobolev}. Thus,
\begin{align*}
 |\langle \cR_{\alpha,A,\chi} u,\phi \rangle | \lesssim \|u\|_{\alpha,2} \|\phi\|_{s',p'} \qquad (\phi \in \C_0^\infty(\R^n))
\end{align*}
and by density $\cR_{\alpha,A,\chi} u$ extends to a functional on $\W^{s',p'}(\R^n)$ as required.

It remains to estimate $\cR_{\beta,B,\chi} g$. In case $0 \leq 2 \beta - \alpha < 1$ and $g \in \W^{2\beta - 2 \alpha + s,p}(\R^{n})$, we can repeat the argument for bounding $I$ and $II$ by replacing $u$ by $g$ and changing the indices of integrability and smoothness in H\"older's inequality accordingly. In this manner,
\begin{align*}
 |\langle \cR_{\beta,B,\chi} g,\phi \rangle | 
 \lesssim \|g\|_p \semi{\phi}_{s',p'} + \semi{g}_{2\beta-2\alpha+s,p}\|\phi\|_{p'}
 \lesssim \|g\|_{2\beta-2\alpha+s,p} \|\phi\|_{s',p'} \qquad (\phi \in \C_0^\infty(\R^n)).
\end{align*}

In the complementary case $2\beta < \alpha$, there is no smoothness of $g$ to be taken advantage of. This, however, can be compensated by the fact $\beta < \alpha /2 < 1/2$. More precisely, we put $\tilde{B}(x,y):= B(x,y) + B(y,x)$ and use the first part of the factorization \eqref{eq:factorization} to write the error term differently as
\begin{align*}
 \langle  \cR_{\beta,B,\chi} g,\phi \rangle
&= \iint_{\R^n \times \R^n} \tilde B(x,y) g(x) \frac{\chi(x) - \chi(y)}{|x-y|^{n+2\beta}} \cl{\phi(y)}   \, \d x \, \d y \\
& := \iint_{\R^n \times \R^n} \tilde B(x,y) g(y) \frac{(\chi(x) - \chi(y) ) \cdot\cl{(\phi(y) - \phi(x))}}{|x-y|^{n+2\beta}}   \, \d x \, \d y \\
& \qquad - \iint_{\R^n \times \R^n} \tilde B(x,y) g(y) \frac{\chi(x) - \chi(y) }{|x-y|^{n+2\beta}} \cl{\phi(y)}   \, \d x \, \d y \\
&:= III + IV,
\end{align*}
where we changed $x$ and $y$ in the second step. Now, our assumption is $g \in \L^t(\R^n)$ with $1/p \leq 1/t < 1/p + s'/n$. We let $1/t + 1/t' = 1$ and obtain from Lemmas~\ref{lem:Fractional Sobolev} and \ref{lem:Fractional Sobolev 2} that the condition on $t$ is precisely to guarantee the continuous inclusions $\W^{s',p'}(\R^n) \subset \W^{\delta,t'}(\R^n) \subset \L^{t'}(\R^n)$ for at least some small $\delta \in (0,1)$. This being said, we use H\"older's inequality and \eqref{eq1:local} with $(s,p)$ replaced by $(2\beta - \delta,t)$ to give
\begin{align*}
|III| 
 \leq \frac{2}{\lambda}\int_{\R^n} |g(y)| \bigg( \int_{\R^n} \frac{|\chi(x) - \chi(y)|^{t}}{|x-y|^{n + (2 \beta - \delta) t}} \, \d x \bigg)^{1/t} \bigg(\int_{\R^n} \frac{|\phi(x) - \phi(y)|^{t'}}{|x-y|^{n+\delta t'}} \, \d x \bigg)^{1/t'} \, \d y
 \lesssim \|g\|_t \|\phi\|_{s',p'}.
\end{align*}
Likewise, for the term $IV$, we use the bound \eqref{eq1:local} with $(s,p)$ replaced by $(2\beta, 1)$ to conclude that 
\begin{align*}
 |IV| 
\lesssim \int_{\R^n}  | g(y)| |\phi(y)| \, \d y \leq \| g \|_{t} \| \phi \|_{t'} \lesssim \|g\|_t \|\phi\|_{s',p'}. &\qedhere
\end{align*}
\end{proof}

\begin{rem}
\label{remark:localmod}
As we mentioned after stating Corollary \ref{thm:local2}, the assumption $g \in \L^{2_{*, \alpha - 2\beta}}(B) \cap \L^{t}(\mathbb{R}^{n})$ for $2 \beta < \alpha$ can be replaced by one \textit{global} assumption $g \in \L^{2_{*, \alpha - 2\beta} + \delta}(\mathbb{R}^{n})$ with $\delta > 0$ in accordance with the result in \cite{KMS}. This follows from a simple modification of the argument above to give the required adaptation of Lemma \ref{lem:local}. We sketch the main idea but leave the precise extensions to the interested reader. The difference arises from the term $\Lop_{\beta,B} g$ so it suffices to see that $ \chi \Lop_{\beta,B} g$ and $\chi f$ belong to the same $\W^{s',p'}(\mathbb{R}^{n})^*$ so that one can apply Proposition \ref{prop:invert}.

If $u$ is a weak solution to \eqref{eq:nonloc}, then automatically
\[ \chi \Lop_{\beta, B} g  \in \W^{\alpha,2}(\mathbb{R}^{n})^{*} \]
by the assumption on $f$, the mapping properties of $\Lop_{\alpha,A}$ and the error term considerations for $\cR_{\alpha,A,\chi} u$. By Corollary \ref{cor:Embdedding Lb small}, 
\[  \chi \Lop_{\beta, B} g \in \W^{\sigma',\tau'}(\mathbb{R}^{n})^{*} \] 
provided that $\frac{1}{q} = \frac{1}{\tau} + \frac{\sigma'-2\beta}{n}$. One can check that there is an admissible choice of $\sigma' < \alpha$ and $\tau' < 2$ when $q = 2_{*, \alpha - 2\beta} + \delta$. By interpolation, we find a line segment $\ell$ connecting $(\sigma',1/\tau')$ to $(\alpha,1/2)$ so that $ \chi \Lop_{\beta, B} g \in \W^{s', p'} (\mathbb{R}^{n}) ^{*}$ for all $(s', 1 / p') \in \ell$. Finally, since $ \chi f \in \L^{t}(\mathbb{R}^{n})$ for all $t \in [1, 2_{*,\alpha} + \delta]$ with $\delta > 0$, there is at least one such $t$ for which we can find $(s',1/p') \in \ell$ with $1/t = 1/p + (2\alpha - s)/n$ so that Lemma \ref{lem:Fractional Sobolev} implies $f \in \W^{s',p'}(\mathbb{R}^{n})^{*}$ with $(s', 1/p')$ as close to $(\alpha, 1/2)$ as desired.  
\end{rem}

\section{An application to fractional parabolic equations}
\label{sec:parabolic}

We demonstrate the flexibility of our approach by a new application to fractional parabolic equations. We shall only treat a particularly interesting special case with connection to non-autonomous maximal regularity, leaving open the establishment of a suitable (full) parabolic analog of Theorem~\ref{thm:main2} and its local version, Theorem~\ref{thm:local}.

We are going to consider the Cauchy problem
\begin{align}
\label{eq:Cauchy}
 \partial_t u(t) + \Lop_{\alpha,A(t)} u(t) &= f(t), \qquad u(0)= 0,
\end{align}
where $f \in \L^2(0,T; \L^2(\R^n))$, $\alpha \in (0,1)$, and for each $t \in [0,T]$ we let $\Lop_{\alpha,A(t)}: \W^{\alpha,2}(\R^n) \to \W^{\alpha,2}(\R^n)^*$ be a fractional elliptic operator as in Section~\ref{sec:Dirichlet form} satisfying the ellipticity condition \eqref{eq:ellip} uniformly in $t$. We recall that the associated sesquilinear forms $\cE_{\alpha,A(t)}$ were defined in \eqref{eq:Lopa}. As for the coefficients 
\begin{align*}
 A(t,x,y):= A(t)(x,y)
\end{align*}
we assume no regularity besides joint measurability in all variables.

Note that we formulated our parabolic problem on $[0,T) \times \R^n$ from the point of view of evolution equations using for, $X$, a Banach space, the space $\L^2(0,T; X)$ of $X$-valued square integrable functions on $(0,T)$ and the associated Sobolev space $\H^1(0,T; X)$ of all $u \in \L^2(0,T;X)$ with distributional derivative  $\partial_t u \in \L^2(0,T; X)$.

\begin{defn}
\label{def:parasol}
Let $f \in \L^2(0,T; \L^2(\R^n))$. A function $u \in \H^1(0,T; \W^{\alpha,2}(\R^n)^*) \cap \L^2(0,T; \W^{\alpha,2}(\R^n))$ is called weak solution to \eqref{eq:Cauchy} if $u(0) = 0$ and
\begin{align}
\label{eq:parasol}
\int_0^T - \langle u, \partial_t \phi \rangle_2 + \cE_{\alpha,A(t)}(u,\phi) \, \d t = \int_0^T \langle f, \phi \rangle_2 \, \d t \qquad (\phi \in \C_0^\infty((0,T) \times \R^n)),
\end{align}
where $\langle \blank, \cdot \rangle_2$ denotes the inner product on $\L^2(\R^n)$.
\end{defn}

\begin{rem}
\label{rem:parasol}
\begin{enumerate}
 \item Since $\W^{\alpha,2}(\R^n)$ is a Hilbert space, the solution space for $u$ above embeds into the continuous functions $\C([0,T]; \L^2(\R^n))$ and hence the requirement $u(0) = 0$ makes sense \cite[Prop.~III.1.2]{Showalter}.
 \item By smooth truncation and convolution $\C_0^\infty((0,T) \times \R^n)$ is dense in $\L^2(0,T; \W^{\alpha,2}(\R^n))$. Thus, the integrated equation \eqref{eq:parasol} precisely means that $u$ satisfies the parabolic equation in \eqref{eq:Cauchy} almost everywhere on $(0,T)$ as an equality in $\W^{\alpha,2}(\R^n)^*$, which contains $\L^2(\R^n)$.
\end{enumerate}
\end{rem}

By a famous result of Lions, the Cauchy problem \eqref{eq:Cauchy} has a unique weak solution $u$ for every $f\in \L^2(0,T; \L^2(\R^n))$. See \cite[p.~513]{Dautray-Lions} and \cite[Thm.~6.1]{Dier-Zacher} for the case of function spaces over the complex numbers. The following self-improvement property is the main result of this section.

\begin{thm}
\label{thm:Cauchy}
Let $f \in \L^2(0,T; \L^2(\R^n))$. Then there exists $\eps > 0$ such that the unique weak solution to \eqref{eq:Cauchy} satisfies 
\begin{align*}
 u \in \H^1(0,T; \W^{\alpha-\eps,2}(\R^n)^*) \cap \L^2(0,T; \W^{\alpha+\eps,2}(\R^n)).
\end{align*}
Moreover, for some $s>\alpha$ and $p>2$ there holds $u \in \W^{\frac{s}{2\alpha},p}(0,T; \L^p(\R^n)) \cap \L^p(0,T; \W^{s,p}(\R^n))$, that is,
\begin{align}
\label{eq:Cauchy Holder estimate}
\begin{split}
&\bigg(\int_0^T \int_{\R^n} |u(t,x)|^p \, \d x \, \d t \bigg)^{1/p}+ 
\bigg(\int_{\R^n} \int_0^T \int_0^T \frac{|u(t,x)-u(s,x)|^p}{|t-s|^{1+sp/(2\alpha)}} \, \d s \, \d t \, \d x \bigg)^{1/p} \\
& \quad + \bigg(\int_0^T \iint_{\R^n \times \R^n} \frac{|u(t,x) - u(t,y)|^p}{|x-y|^{n+sp}} \, \d x \, \d y \, \d t \bigg)^{1/p}
\lesssim \e^T \bigg(\int_0^T \int_{\R^n} |f(t,x)|^2 \, \d x \, \d t \bigg)^{1/2}.
\end{split}
\end{align}
The values of $\eps, s, p$ and the implicit constant in \eqref{eq:Cauchy Holder estimate} depend only on $\lambda,n, \alpha$.
\end{thm}

\begin{rem}
\begin{enumerate}
 \item Since $sp > 2\alpha$, the boundedness of the second integral in \eqref{eq:Cauchy Holder estimate} entails, in particular, $u \in \C^\gamma([0,T]; \L^p(\R^n))$ with H\"older exponent $\gamma = \frac{sp}{2\alpha} - 1$, see fore example \cite[Cor.~26]{Simon}.
 \item The largest possible value $\eps = \alpha$ with $\W^{0,2}(\R^n):=\L^2(\R^n)$ would mean \emph{maximal regularity} because all three functions in the parabolic equation were in the same space $\L^2(0,T; \L^2(\R^n))$. See \cite{MaxRegSurvey} for further background and (counter-)examples.
\end{enumerate}
\end{rem}

For the proof, we shall apply the same scheme as in the stationary case, see Sections~\ref{sec:Dirichlet form} and \ref{sec:higher diff}.

\subsection{Definition of the parabolic Dirichlet form}

One of the immediate challenges in moving from the elliptic operator to the parabolic operator is the lack of coercivity of the operator $\partial_t + \Lop_{\alpha,A(t)}$. However, we can rely on the \emph{hidden coercivity} introduced in this context in \cite{Dier-Zacher}. This requires us to study the fractional parabolic equation for $t \in \R$ first, that is,
\begin{align*}
 \partial_t u(t) + \Lop_{\alpha,A(t)} u(t) &= f(t),
\end{align*}
where weak solutions are in the sense of Definition~\ref{def:parasol}, but by replacing $(0,T)$ with $\R$ and of course removing the initial condition. Note that we can simply extend the coefficients by $A(t,x,y) :=1$ if $t \notin [0,T]$ since we are not assuming any regularity. 

For simplicity, put $H:= \L^2(\R^n)$ and $V:= \W^{\alpha,2}(\R^n)$. Let $\cF$ be the Fourier transform in $t$ on the vector-valued space $\L^2(\R; H)$ and define the \emph{half-order time derivative} $\dhalf$ and the \emph{Hilbert transform} $\HT$ through the Fourier symbols $|\tau|^{1/2}$ and $- \i \sgn(\tau)$, respectively. They are crafted to factorize $\partial_t = \dhalf \HT \dhalf$. Next, we write $\H^{1/2}(\R; H)$ for the Hilbert space of all $u \in \L^2(\R;H)$ such that $\dhalf u \in \L^2(\R; H)$ and define the parabolic \emph{energy space}
\begin{align*}
 \IE:= \H^{1/2}(\R; H) \cap \L^2(\R; V)
\end{align*}
equipped with the Hilbertian norm $\|u\|_\IE := \big(\|u\|_{\L^2(\R; V)}^2 + \|\dhalf u\|_{\L^2(\R; H)}^2\big)^{1/2}$. It allows one to define $1+ \partial_t + \Lop_{\alpha,A(t)}$ as a bounded operator $\IE \to \IE^*$ via
\begin{align}
\label{sesqui para}
 \langle (1+ \partial_t + \Lop_{\alpha,A(t)})u ,v \rangle := \int_{\R} \langle u, v \rangle_2 + \langle \HT \dhalf u, \dhalf v \rangle_2 + \cE_{\alpha,A(t)}(u,v) \, \d t,
\end{align}
where $\langle \cdot \,, \cdot \rangle_2$ denotes the inner product on $H = \L^2(\R^n)$. We state our substitute for Lemma~\ref{lem:Lax-Milgram} in the parabolic case. It is an extension of Theorem~3.1 in \cite{Dier-Zacher}.

\begin{lem}
\label{lem:Lax-Milgram para}
The operator $1 + \partial_t + \Lop_{\alpha,A(t)}: \IE \to \IE^*$ is bounded and invertible. Its norm and the norm of its inverse can be bounded only in terms of $\lambda$. Moreover, given $f \in \L^2(\R; H)$, $u:= (1 + \partial_t + \Lop_{\alpha,A(t)})^{-1} f$ is a weak solution to $\partial_t u + \Lop_{\alpha,A(t)}u = f - u$ on $\R^{1+n}$.
\end{lem}

\begin{proof}
The $\IE \to \IE^*$ boundedness of $1 + \partial_t + \Lop_{\alpha,A}$ is clear by definition. Next, for the invertibility, the form 
\begin{align*}
 a_\delta(u,v) := \int_{\R} \langle u, (1+\delta\HT) v \rangle_2 + \langle \HT \dhalf u, \dhalf (1+\delta\HT)v \rangle_2 + \cE_{\alpha,A(t)}(u,(1+\delta\HT)v) \, \d t
\end{align*}
for $u,v\in \IE$, is bounded and satisfies an accretivity bound for $\delta>0$ sufficiently small, for example $\delta := \lambda^2/2$. Indeed, from boundedness and ellipticity of $\cE_{\alpha,A(t)}$ uniformly in $t$ (see Section~\ref{sec:Dirichlet form}) and the fact that the Hilbert transform is $\L^2$-isometric and skew-adjoint,
\begin{align*}
\Re a_\delta(u,u) \geq \|u\|_{\L^2(\R; H)}^2 + \delta \|\dhalf u \|_2^2 + (\lambda-\lambda^{-1}\delta ) \int_{\mathbb{R}} \semi{u(t, \cdot)}_{\alpha,2}^{2} \, \d t \geq \frac{\lambda^2}{2} \|u\|_{\IE}^{2}. 
\end{align*}
As  
\begin{align*}
\dual{(1 + \partial_t + \Lop_{\alpha,A(t)}) u}{ (1+\delta\HT)v} = a_\delta(u,v), \qquad (u,v\in \IE), 
\end{align*}
and since $(1+\delta^2)^{-1/2}(1+\delta\HT)$ is isometric on $\IE$ as is seen using its symbol $(1+\delta^2)^{-1/2}(1- \i \delta \sgn\tau )$, it follows from the Lax-Milgram lemma that $1 + \partial_t + \Lop_{\alpha,A(t)}$ is invertible from $\IE$ onto $\IE^*$. Finally, given $f \in \L^2(\R; H) \subset \IE$ we can define $u:= (1 + \partial_t + \Lop_{\alpha,A(t)})^{-1} f$ and have by definition
\begin{align*}
 \int_{\R}  \langle \HT \dhalf u, \dhalf v \rangle_2 + \cE_{\alpha,A(t)}(u,v) \, \d t = \int_{\R} \langle f-u, v \rangle_2 \, \d t \qquad (v \in \IE).
\end{align*}
Since for $v \in \C_0^\infty(\R \times \R^n)$ we can undo the factorization $\langle \HT \dhalf u, \dhalf v \rangle_2 = - \langle u, \partial_t v \rangle$, we see that $u$ is a weak solution to $\partial_t u + \Lop_{\alpha,A(t)}u = f - u$.
\end{proof}

\begin{rem}
Skew-adjointness of the Hilbert transform and ellipticity of each sesquilinear form $\cE_{\alpha,A(t)}$ yield $\Re \langle (\partial_t + \Lop_{\alpha,A(t)})u ,u \rangle \geq 0$ for every $u \in \IE$ and by the previous lemma $1+(\partial_t + \Lop_{\alpha,A(t)}) : \IE \to \IE^*$ is invertible. By definition, this means that $\partial_t + \Lop_{\alpha,A(t)}$ can be defined as a \emph{maximal accretive operator} in $\L^2(\R^{1+n})$ with maximal domain $\mathbb{D}:=\{u \in \IE : (\partial_t + \Lop_{\alpha,A(t)}) u \in \L^2(\R^{1+n})\}$.
\end{rem}

In order to proceed, we need to link the parabolic energy space $\IE$ and the sesquilinear form on the right-hand side of \eqref{sesqui para} with a Dirichlet form on fractional Sobolev spaces as in Section~\ref{sec:Dirichlet form}. To this end, note that for $u, v \in \L^2(\R; H)$ we obtain from Plancherel's theorem applied to the integral in $s$,
\begin{align*}
 \iint_{\R \times \R} \frac{\langle u(s+h)-u(s), v(s+h)-v(s)\rangle_2}{|h|^2} \, \d s \, \d h
&= \iint_{\R \times \R} \frac{|\e^{-\i h \tau} - 1|^{2}}{|h|^2} \langle \cF{u}(\tau), \cF{v}(\tau) \rangle_2 \, \d \tau \, \d h \\
&= 2 \pi \int_{\R} \langle \dhalf u(t), \dhalf v(t) \rangle_2 \, \d t,
\end{align*}
where in the second step we evaluated the well-known integral in $h$ to $2 \pi |\tau|$. This calculation is understood in the sense that for $u = v$ the left-hand side is finite if and only if the right-hand side is defined and finite and if both $u$ and $v$ have this property, then equality above holds true. Consequently, $\partial_t + \Lop_{\alpha,A(t)}$ is the operator associated with the \emph{parabolic Dirichlet form}
\begin{align*}
 \cP_{\alpha, A(t)}(u,v)
&:= \int_{\R} \langle \HT \dhalf u, \dhalf v \rangle_2 + \cE_{\alpha,A(t)}(u,v) \, \d t \\
&= \frac{1}{2 \pi}\int_{\R^n} \iint_{\R \times \R} \frac{(\HT u(t,x)- \HT u(s,x)) \cdot \cl{(v(t,x)-v(s,x))}}{|t-s|^2} \, \d s \, \d t \, \d x \\
&\quad + \int_{\R} \iint_{\R^n \times \R^n} A(t,x,y) \frac{(u(t,x) - u(t,y))\cdot \cl{(v(t,x)-v(t,y))}}{|x-y|^{n+2\alpha}} \, \d x \, \d y \, \d t,
\end{align*}
defined so far for $u,v \in \IE$. Here, $\HT u(\blank,x)$ is understood as the Hilbert transform of $u(\blank,x) \in \L^2(\R)$ for almost every fixed $x \in \R^n$.

\subsection{Analysis of the parabolic Dirichlet form}
\label{sec:Dirichlet para}

The definition of the parabolic Dirichlet form determines the spaces `nearby' $\IE$ to look at: For $s \in (0,1) \cap (0,2 \alpha)$ and $p \in (1,\infty)$ we let $\IW^{s,p}_\alpha(\R^{1+n})$ consist of all functions $u \in \L^p(\R^{1+n})$ with finite semi-norm
\begin{align*}
\Isemi{u}_{s,p}
:= \bigg(\int_{\R^n} \iint_{\R \times \R} \frac{|u(t,x)-u(s,x)|^p}{|t-s|^{1+sp/(2\alpha)}} \, \d s \, \d t \, \d x 
+ \int_{\R} \iint_{\R^n \times \R^n} \frac{|u(t,x) - u(t,y)|^p}{|x-y|^{n+sp}} \, \d x \, \d y \, \d t \bigg)^{1/p}
\end{align*}
and put $\|\cdot\|_{\IW^{s,p}_\alpha(\R^{1+n})} := \|\cdot\|_p + \Isemi{\cdot}_{s,p}$. Again, smooth truncation and convolution yields density of $\C_0^\infty(\R^{1+n})$ in any of these spaces. Often we shall write more suggestively 
\begin{align*}
 \IW^{s,p}_\alpha(\R^{1+n}) = \W^{\frac{s}{2\alpha},p}(\R; \L^p(\R^n)) \cap \L^p(\R; \W^{s,p}(\R^n)),
\end{align*}
where the vector-valued fractional Sobolev spaces are defined as their scalar-valued counterpart upon replacing absolute values by norms. But as $\W^{\frac{s}{2\alpha},p}(\R; \L^p(\R^n)) = \L^p(\R^n; \W^{\frac{s}{2\alpha},p}(\R))$ in virtue of Tonelli's theorem, all fractional Sobolev embeddings stated for the scalar-valued space $\W^{\frac{s}{2\alpha},p}(\R)$ remain valid for $\W^{\frac{s}{2\alpha},p}(\R; \L^p(\R^n))$. Note the scaling in the spaces $\IW^{s,p}_\alpha(\R^{1+n})$ adapted to the fractional parabolic equation: one time derivative accounts for $2 \alpha$ spatial derivatives.

By what we have seen before, $\IW^{\alpha,2}_\alpha(\R^{1+n}) = \IE$ up to equivalent norms and hence $1+\partial_t + \Lop_{\alpha,A(t)}$ is invertible from that space onto its anti-dual by Lemma~\ref{lem:Lax-Milgram para}. The following mapping properties are then proved by H\"older's inequality exactly as their elliptic counterpart, Lemma~\ref{lem:1+L bounded}, on making the additional observation that $\HT: \W^{s,p}(\R) \to \W^{s,p}(\R)$ is bounded. Indeed, this is immediate from the equivalent norm \eqref{BesovNorm} on $\W^{s,p}(\R)$ since the Hilbert transform commutes with convolutions and is bounded on $\L^p(\R)$.

\begin{lem}
\label{lem:1+L bounded para}
Let $s,s' \in (0,1)$ and $p,p' \in (1,\infty)$ satisfy $s + s' = 2 \alpha$ and $1/p + 1/p' = 1$. Then $1 + \partial_t + \Lop_{\alpha,A(t)}$ extends from $\C_0^\infty(\R^n)$ by density to a bounded operator $\IW^{s,p}_\alpha(\R^{1+n}) \to \IW^{s',p'}_\alpha(\R^{1+n})^*$. 
\end{lem}

\begin{rem}
The extensions obtained above are also denoted by $1 + \partial_t + \Lop_{\alpha,A}$ and a comment analogous to Remark~\ref{rem:1+L bounded} applies.
\end{rem}

Hence, the only ingredient missing in our recipe for self-improvement is the complex interpolation identity replacing \eqref{eq:interpolSob}. This can be obtained from \cite{Dachkovski} as follows. We define the vector of anisotropy $\bv$ and the mean smoothness $\gamma$ by
\begin{align*}
 \bv:= \big(\tfrac{2 \alpha(1+n)}{n+2\alpha}, \tfrac{1+n}{n+2\alpha}, \ldots,\tfrac{1+n}{n+2\alpha} \big)  \in \R^{1+n}, \qquad  \gamma:= \tfrac{(1+n)}{n+2\alpha} s \in (0,1) \quad \text{for $s \in (0,1) \cap (0,2\alpha)$}.
\end{align*}
Then, \cite[Thm.~6.2]{Dachkovski} identifies $\IW^{s,p}_\alpha(\R^{1+n})$ up to equivalent norms with the \emph{anisotropic Besov space} $\B^{\gamma,\bv}_{p,p}(\R^{1+n})$. In turn, this space is defined in \cite{Dachkovski} exactly as the ordinary Besov space $\B^{\gamma}_{p,p}(\R^{1+n})$ in Section~\ref{sec:Notation}, upon replacing the scalar multiplication $2^j x = (2^j x_0,\ldots, 2^j x_n)$ on $\R^{1+n}$ by the anisotropic multiplication $2^{\bv j} x := (2^{\bv_0 j} x_0,\ldots 2^{\bv_{n} j} x_n)$, where $j \in \R$ and subscripts indicate coordinates of $(n+1)$-vectors, and the Euclidean norm $|x|$ by the anisotropic norm $|x|_{\bv}$ defined as the unique positive number $\sigma$ such that $\sum_{j} x_j^2/\sigma^{2 \bv_j} = 1$. With these modifications, $\B^{\gamma,\bv}_{p,p}(\R^{1+n})$ is the collection of all $ u \in \L^p(\R^{1+n})$ with finite norm
\begin{align*}
 \|u\|_{\B^{\gamma, \bv}_{p,p}(\R^{1+n})} := \bigg(\sum_{j = 0}^\infty 2^{j \gamma p} \|\phi_j \ast u\|_p^p \bigg)^{1/p} < \infty.
\end{align*}
Note that this norm now reads exactly as the one in \eqref{BesovNorm} on the anisotropic space $\B^{\gamma}_{p,p}(\R^{1+n})$ because the anisotropy $\bv$ is only present in the now anisotropic dyadic decomposition $1 = \sum_{j=0}^\infty \cF(\phi_j)(\xi)$. With this particular structure of the norms, complex interpolation works by abstract results exactly as outlined before in Section~\ref{sec:Dirichlet form}, see again \cite[Thm.~6.4.5(6)]{Bergh-Loefstroem} and \cite[Cor.~4.5.2]{Bergh-Loefstroem}. Thus, we have
\begin{align*}
 \big[\IW^{s_0,p_0}_\alpha(\R^n),  \IW^{s_1,p_1}_\alpha(\R^{1+n})\big]_\theta = \IW^{s,p}_{\alpha}(\R^n)
\end{align*}
for $p_0, p_1 \in (1,\infty)$, $s_0,s_1 \in (0,1) \cap (0,2\alpha)$ and the analogous identity for the anti-dual spaces both up to equivalent norms with $p,s$ given as before by $\tfrac{1}{p} = \tfrac{1-\theta}{p_0} + \tfrac{\theta}{p_1}$ and $s = (1-\theta)s_0 + \theta s_1$. We do not insist on uniformity of the equivalence constants as in Section~\ref{sec:Dirichlet form} and leave the care of checking it to interested readers.

This interpolation identity and Lemma~\ref{lem:1+L bounded para} set the stage to apply {\v{S}}ne{\u{\ii}}berg's result as in the proof of Proposition~\ref{prop:invert} to deduce

\begin{prop}
\label{prop:invert para}
Fix any line $\ell$ passing through $(\alpha,1/2)$ in the $(s, 1/p)$-plane. There exists $\eps > 0$ depending on $\ell, \lambda, n$, such for $(s,1/p) \in \ell$ with $|s-\alpha|, |p-2| < \eps$ and $s',p'$ satisfying $s+s' = 2\alpha$ and $1/p + 1/p' = 1$, the operator
\begin{align*}
 1 + \partial_t +  \Lop_{\alpha,A(t)} : \IW^{s,p}_\alpha(\R^{1+n}) \to \IW^{s',p'}_\alpha(\R^{1+n})^*
\end{align*}
is invertible and the inverse agrees with the one obtained for $s=\alpha$, $p=2$ on their common domain of definition.
\end{prop}

\subsection{Higher differentiability and integrability result}

We still need a lemma making Proposition~\ref{prop:invert para} applicable in the $\L^2$-setting of our main result.

\begin{lem}
\label{lem:IW embedding}
Suppose $s \in (\alpha, 2 \alpha)$, $p \in [2,\infty)$ and let $s+s' = 2\alpha$, $1/p + 1/p' = 1$. If $2/p \geq 1 - s'/n$, then $\L^2(\R; \L^2(\R^n)) \subset \IW_\alpha^{s',p'}(\R^{1+n})^*$ with continuous inclusion.
\end{lem}

\begin{proof}
Since $p' s' < 2\alpha < 2 \leq n$ by assumption, we can infer from Lemma~\ref{lem:Fractional Sobolev} the continuous embedding 
\begin{align*}
 \W^{s',p'}(\R^n) \subset \L^q(\R^n) \qquad (\tfrac{1}{p'} - \tfrac{s'}{n} \leq \tfrac{1}{q} \leq \tfrac{1}{p'}).
\end{align*}
Likewise, by the vector valued analog of Lemma~\ref{lem:Fractional Sobolev} (see the beginning of Section~\ref{sec:Dirichlet para}) we have 
\begin{align*}
 \W^{\frac{s'}{2\alpha},p'}(\R; \L^{p'}(\R^n)) \subset \L^r(\R; \L^{p'}(\R^n)) \qquad (\tfrac{1}{p'} - \tfrac{s'}{2\alpha} \leq \tfrac{1}{r} \leq \tfrac{1}{p'})
\end{align*}
Now, the additional condition $2/p \geq 1 - s'/n$ along with $2\alpha < n$ precisely guarantees that we can take $q=p=r$ and therefore 
\begin{align*}
 \IW^{s',p'}_\alpha(\R^{1+n}) = \W^{\frac{s'}{2\alpha},p'}(\R; \L^{p'}(\R^n)) \cap \L^{p'}(\R; \W^{s',p'}(\R^n)) \subset \L^{p}(\R; \L^{p'}(\R^n)) \cap \L^{p'}(\R; \L^{p}(\R^n)).
\end{align*}
Taking into account the convex combinations $\frac{1}{2} = \frac{1-\theta}{p} + \frac{\theta}{p'} =  \frac{1-\theta}{p'} + \frac{\theta}{p}$ for $\theta = \frac{1}{2}$, standard embeddings for mixed Lebesgue spaces imply that the right-hand space is continuously included in $\L^2(\R; \L^2(\R^n))$, see for example \cite[Thm.~5.1\&5.2]{Bergh-Loefstroem}. The claim follows by duality with respect to the inner product on $\L^2(\R; \L^2(\R^n))$. 
\end{proof}

With this at hand, we are ready to give the

\begin{proof}[Proof of Theorem~\ref{thm:Cauchy}]
Let $f \in \L^2(0,T; \L^2(\R^n))$. Since uniqueness is known, only existence of a weak solution to \eqref{eq:Cauchy} with the stated properties is a concern. To this end, we shall argue as in \cite{Dier-Zacher} by restriction from the real line, where we know how to improve regularity. 

We extend $A(t,x,y) :=1$ and $f(t):=0$ for $t \notin [0,T]$. Then, $g(t):=\e^{-t}f(t) \in \L^2(\R; \L^2(\R^n))$ and thus Lemma~\ref{lem:Lax-Milgram para} furnishes 
\begin{align*}
 v:=(1 + \partial_t + \Lop_{\alpha,A})^{-1} g \in \IW^{\alpha,2}_\alpha(\R^{1+n}),
\end{align*}
which is a weak solution to
\begin{align*}
 \partial_t v(t) + \Lop_{\alpha,A(t)}v(t) = e^{-t}f(t) - v(t) \qquad (t \in \R).
\end{align*}
In particular, $v$ is a continuous function on $\R$ with values in $\L^2(\R^n)$ (see Remark~\ref{rem:parasol} (i)). We claim $v(0)=0$. Indeed, $t \mapsto \|v(t)\|_2^2$ is absolutely continuous with derivative $\frac{\mathrm{d}}{\mathrm{d} t} \|v(t)\|_2^2 = 2 \Re \langle \partial_tv(t), v(t) \rangle$, where $\langle \blank, \cdot \rangle$ denotes the $\W^{\alpha,2}(\R^n)^*$-$\W^{\alpha,2}(\R^n)$ duality \cite[Prop.~1.2]{Showalter}. By \eqref{eq:lower bound sesqui},
\begin{align*}
 \lambda \int_{-\infty}^0 \|v\|_{\alpha,2}^2 \d t
\leq \Re \int_{-\infty}^0 \langle v+ \Lop_{\alpha,A(t)} v, v \rangle \, \d t
= - \Re \int_{-\infty}^0 \langle \partial_t v, v \rangle \, \d t
= -\frac{1}{2} \|v(0)\|_2^2,
\end{align*}
where we have used the equation for $v$ along with $f(t) = 0$ for $t\in (-\infty,0)$ in the second step. Thus, $\|v(0)\|_2 = 0$. The upshot is that the restriction of $\e^t v(t)$ to $[0,T]$ is the unique weak solution $u$ to the Cauchy problem \eqref{eq:Cauchy} and it remains to prove the additional regularity.

Let $s>\alpha$, $p>2$ sufficiently close to $\alpha$, $2$, so that we have both Lemma~\ref{lem:IW embedding} and Proposition~\ref{prop:invert para} at our disposal. Defining $s'$ and $p'$ as usual, the former guarantees $g \in \IW_\alpha^{s',p'}(\R^{1+n})^*$ and thus the latter yields $v \in \IW^{s,p}_\alpha(\R^{1+n})$. As we have $u(t) = \e^{t} v(t)$ for $t \in [0,T]$, restricting to $[0,T]$ readily yields that the left-hand side of \eqref{eq:Cauchy Holder estimate} is controlled by 
\begin{align*}
 \e^T(\|v\|_{p} + \Isemi{v}_{s,p}) 
\lesssim \e^T \|g\|_{\IW_\alpha^{s',p'}(\R^{1+n})^*}
\lesssim \e^T \|g\|_{\L^2(\R; \L^2(\R^n))}
\lesssim \e^T \|f\|_{\L^2(0,T; \L^2(\R^n))}
\end{align*}
as claimed.

Repeating the same argument with $s> \alpha$ and $p=2$ reveals $v \in \IW^{s,2}_\alpha(\R^{1+n})$ and in particular $u \in \L^2(0,T; \W^{\alpha+\eps,2}(\R^n))$, where $\eps:=s-\alpha > 0$. By H\"older's inequality this also implies $u \in \L^2(0,T; \W^{\alpha-\eps,2}(\R^n)^*)$. Moreover, from the equation for $u$ since $\Lop_{\alpha,A(t)}: \W^{\alpha+\eps,2}(\R^n) \to \W^{\alpha-\eps,2}(\R^n)^*$ is bounded by $\lambda^{-1}$ uniformly in $t$ due to Lemma~\ref{lem:1+L bounded}, we deduce
\begin{align*}
\bigg|\int_0^T - \langle u , \partial_t \phi \rangle_2  \, \d t \bigg|
 \leq \int_0^T \|f(t)\|_2 \|\phi(t)\|_2 + \lambda^{-1} \|u(t)\|_{\alpha+\eps,2} \|\phi(t)\|_{\alpha-\eps,2} \, \d t
\end{align*}
for all $\phi \in \C_0^\infty((0,T) \times \R^n)$. By density, see Remark~\ref{rem:parasol}, this remains true for $\phi \in \H^1(0,T; \W^{\alpha-\eps,2}(\R^n))$ and we conclude $u \in \H^1(0,T; \W^{\alpha-\eps,2}(\R^n)^*)$ as required.
\end{proof}
\def\cprime{$'$} \def\cprime{$'$} \def\cprime{$'$}

\end{document}